\documentclass[12pt]{amsart}
\usepackage[utf8]{inputenc}
\usepackage{amsmath}
\usepackage{amssymb}
\usepackage{tikz}
\usepackage{color}
\usepackage{xcolor}
\usepackage{float}
\usepackage{caption}
\usepackage{wasysym}

\usepackage{mathrsfs}
\usepackage{makecell}

\usepackage{geometry}
\usepackage{cite}
\makeatletter
\@namedef{subjclassname@2020}{%
  \textup{2020} Mathematics Subject Classification}
\makeatother

\theoremstyle{definition}

\newtheorem{theorem}{Theorem}
\newtheorem{corollary}[theorem]{Corollary}
\newtheorem{problem}[theorem]{Problem}
\newtheorem{proposition}[theorem]{Proposition}
\newtheorem{definition}[theorem]{Definition}

\newcommand{\supp}[1]{\mathrm{supp}\left(#1\right)}
\newcommand\Wp{\mathcal{W}_p}

\title[Isometric rigidity of $\mathcal{W}_p(X)$ - the graph metric case]{Isometric rigidity of Wasserstein spaces: \\ the graph metric case}
\author[Gergely Kiss]{Gergely Kiss}
\address{Gergely Kiss, Alfr\'ed R\'enyi Institute of Mathematics -- Eötvös Loránd Research Network\\ Re\'altanoda u. 13-15.\\ Budapest H-1053\\ Hungary.}
\email{kiss.gergely@renyi.hu}

\author[Tam\'as Titkos]{Tam\'as Titkos}
\address{Tam\'as Titkos, Alfr\'ed R\'enyi Institute of Mathematics -- Eötvös Loránd Research Network\\ Re\'altanoda u. 13-15.\\
Budapest H-1053\\ Hungary\\ and BBS University of Applied Sciences\\ Alkotm\'any u. 9.\\
Budapest H-1054\\ Hungary.}
\email{titkos.tamas@renyi.hu}
\thanks{Corresponding author: Tamás Titkos, titkos.tamas@renyi.hu}
\thanks{G. Kiss was supported by Premium Postdoctoral Fellowship of the Hungarian Academy of Sciences and by the Hungarian National Research, Development and Innovation
Office - NKFIH (grant no. K124749). T. Titkos was supported by the Hungarian National Research, Development and Innovation Office - NKFIH (grant no. PD128374 and grant no. K134944), by the János Bolyai Research Scholarship and the Momentum Program No. LP2021-15/202 of the Hungarian Academy of Sciences, and by the ÚNKP-20-5-BGE-1 New National Excellence Program of the Ministry of Innovation and Technology.}

\begin{document}


\begin{abstract}
The aim of this paper is to prove that the $p$-Wasserstein space $\mathcal{W}_p(X)$ is isometrically rigid for all $p\geq1$ whenever $X$ is a countable graph metric space. As a consequence, we obtain that for every countable group ${H}$ and any $p\geq1$ there exists a $p$-Wasserstein space whose isometry group is isomorphic to ${H}$.
\end{abstract}
\maketitle
\section{Introduction}
Due to its deep impact on both pure and applied sciences, one of the most intensively studied metric spaces nowadays is the so-called $p$-Wasserstein space $\mathcal{W}_p(X)$: the collection of Borel probability measures on a complete separable metric space $(X,\varrho)$ with finite $p$-th moment, endowed with a transport related metric $d_p$, which is calculated by means of optimal couplings and the $p$-th power of the underlying distance $\varrho$. We mention here only three comprehensive textbooks \cite{AG,Santambrogio,Villani}, more references and precise definitions  will follow later. In this paper we consider those $p$-Wasserstein spaces whose underlying metric space $(X,\varrho)$ is a graph metric space. This class contains many important metric spaces, just to mention a few: any countable set with the discrete metric; the set of natural numbers and the set of integers with the usual $|\cdot|$-distance; $d$-dimensional lattices endowed with the $\mathit{l}_1$-metric or the $\mathit{l}_{\infty}$-metric (for the relevance of these metrics in pattern recognition see e.g. \cite{Rhodes}); finite strings with the Hamming distance (as it was mentioned in \cite{DP-M-T-L} in connection with the quantum $1$-Wasserstein distance, the classical $1$-Wasserstein distance with respect to the Hamming metric is called Ornstein’s distance, and was first considered in \cite{Ornstein}); and finite regular trees (see the very recent manuscript \cite{Jiradilok}).

When working with a structure, the most fundamental and natural task is to explore its transformations and symmetries.
In the case of metric spaces, such symmetries are \emph{isometries}, that is, distance preserving bijections. In the recent past, many authors investigated isometries of various important metric spaces of probability measures \cite{bertrand-kloeckner-2016,DolinarKuzma,DolinarMolnar,LP,JMAA,TAMS,HIL,Kloeckner-2010,Levy,S-R,Virosztek} . To mention a few, in \cite{Levy} Molnár explored the structure of isometries of the space of distribution functions with respect to the Lévy distance. Later Gehér and the second author generalised his result to the Lévy-Prokhorov metric in \cite{LP}. Namely, it was shown that if the space $\mathcal{P}(X)$ of all Borel probability measures on a real Banach space $(X,\|\cdot\|)$ is endowed with the Lévy-Prokhorov metric $d_{\mathrm{LP}}$, then the isometry group of $\big(\mathcal{P}(X),d_{\mathrm{LP}}\big)$ is isomorphic to the isometry group of the underlying space $X$. Bertrand and Kloeckner showed that a similar phenomenon occurs when one considers a $2$-Wasserstein space built on a negatively curved metric space: each isometry of the space of measures is the push-forward of an isometry of the underlying space. This phenomenon is called \emph{isometric rigidity}. Finally, we mention a very recent rigidity result, Santos-Rodríguez proved that $p$-Wasserstein spaces built on compact rank one symmetric spaces are all isometrically rigid \cite{S-R} if $p>1$.\\

Our main result is \textbf{Theorem \ref{t:main}}, where we prove that $p$-Wasserstein spaces over graph metric spaces are all isometrically rigid. As a consequence, in \textbf{Corollary \ref{c:aut}} we will conclude that for every countable group $G$ and for every $p\geq 1$ there exists a $p$-Wasserstein space whose isometry group is isomorphic to $G$.\\

Before going into the details, we make some comments on rigidity. It comes easy to say that these rigidity results are not surprising because of the intimate connection between $d_p$ and $\varrho$. It is well known that if $p\geq1$ then the distance between any two Dirac measures equals to the distance of their supporting points, and every measure can be approximated by convex combinations of Dirac measures. In other words, $\mathcal{W}_p(X)$ contains an isometric copy of $X$, and the convex hull of this copy is dense in $\mathcal{W}_p(X)$. Moreover, $\mathcal{W}_p(X)$ inherits many nice properties of $X$, e.g. completeness, compactness, existence of geodesics. So one may have the impression that although the $p$-Wasserstein space $\mathcal{W}_p(X)$ is much bigger than $X$ (see e.g. \cite{Kloeckner-Hausdorff} for many interesting results), the strong connection between the metrics does not allow $\mathcal{W}_p(X)$ to have more symmetries than $X$ has. A possible sketch of proof is this:
\begin{itemize}
    \item[\underline{Step 1.}] Prove that an isometry $\Phi:\mathcal{W}_p(X)\to\mathcal{W}_p(X)$ leaves the set of Dirac masses invariant. Once it is done, one can conclude that the action on Dirac masses is generated by an isometry $f:X\to X$, that is, $\Phi(\delta_x)=\delta_{f(x)}$ for all $x\in X$.
    \item[\underline{Step 2.}] Prove that this action extends to a set $\mathcal{S}$ of finitely supported probability measures, that is, $\Phi\Big(\sum\limits_{j\in J}\lambda_j\delta_{x_j}\Big)=\sum\limits_{j\in J}\lambda_j\delta_{f(x_j)}$ for all $\sum\limits_{j\in J}\lambda_j\delta_{x_j}\in\mathcal{S}$.
    \item[\underline{Step 3.}] Show that $\mathcal{S}$ is dense in $\mathcal{W}_p(X)$. Since $\Phi$ is continuous, $\Phi$ must be the push-forward of $f^{-1}$, where $f$ is the above defined isometry.
\end{itemize}

The problem with this sketch is that it does not work in general. And even if it works, these seemingly easy steps can be nontrivial. For example, {Step 1} fails if $p=1$ and $X=[0,1]$. In that case, there exists an isometry $j$ (called \emph{flip}) which is mass-splitting, i.e. which does not leave the set of Dirac masses invariant: $j(\delta_t)=t\delta_0+(1-t)\delta_1$ for all $t\in[0,1]$, for more details see \cite[Section 2.1]{TAMS}. {Step 2} can fail (even if Step 1 can be done) as it was shown by Kloeckner in \cite{Kloeckner-2010}: if $p=2$ and $X=\mathbb{R}$, then there exist a flow of strangely behaving isometries. These isometries leave all Dirac masses fixed, but they differ from the identity of $\mathcal{W}_2(\mathbb{R})$, for more details see \cite[Section 5.1]{Kloeckner-2010}. We mention that all these strange isometries disappear once we modify the value of $p$: it was proved in \cite{TAMS} that if $p\neq1$, then $\mathcal{W}_p\big([0,1]\big)$ is isometrically rigid, and similarly, if $p\neq2$, then $\mathcal{W}_p(\mathbb{R})$ is isometrically rigid. Furthermore, Gehér et al. showed in \cite[Section 2]{HIL} that for every $p\geq 1$ there exists a compact metric space $X$ such that $\mathcal{W}_p(X)$ admits mass-splitting isometries.

Summarising the above comments, we can say that isometric rigidity of $p$-Wasserstein spaces is a quite regular phenomenon (only a few non-rigid example is known), but there is no general recipe which helps to decide whether a space is rigid or not.\\

\noindent\textbf{Acknowledgements:} We would like to thank the referee for the careful reading of the manuscript and the constructive comments that helped us to improve the presentation.

\section{Technical preliminaries}
First we fix the terminology. Given a metric space $(Y,d)$, we call a bijection $f:Y\to Y$ an \emph{isometry} if $d(f(x),f(y))=d(x,y)$ holds for all $x,y\in Y$. The \emph{isometry group} of $(Y,d)$ will be denoted by $\mathrm{Isom}(Y,d)$. For two groups $H_1,H_2$ the symbol $H_1 \cong H_2$ means that they are isomorphic.

In this paper $G(X,E)$ always denotes a graph with a countable vertex set $X$ and edge set $E$. Two different vertices $x,y\in X$ are called adjacent if there exists an edge $e\in E$ which joins them. A \emph{path} between two vertices $x,y\in X$ is a finite sequence of distinct edges which joins a sequence $x_0=x, x_1,\dots,x_k=y$ of distinct adjacent vertices. The length of such a path is the number of edges that the path contains. The graph $G(X,E)$ is called connected if every two different vertices can be connected by a path. The vertex set $X$ of a connected graph can always be endowed with a metric: for $x,y\in X$ we say that the \emph{shortest path distance} of $x$ and $y$ is the minimum number of $k$ such that there exists a path between $x$ and $y$ of length $k$. Now we can define the central notion of this paper: we say that a countable metric space $(X,\varrho)$ is a \emph{graph metric space} if there exists a connected graph $G(X,E)$, such that the shortest path distance in the graph coincides with the distance $\varrho$ in $X$. That is, $\varrho(x,y)$  equals to the minimum number of edges in a path in $G(X,E)$ between $x$ and $y$. Since the existence of loops and multiple edges do not change the length of the shortest path, one can assume that the graph in question is simple.\footnote{The following characterization of graph metric spaces was proved in \cite{CK}: a countable metric space $(X, \varrho)$ is a graph metric space if and only if the distance between every two points of $X$ is an integer, and if $a, b \in X$ and $\varrho(a, b) \geq2$ then there exists a point $x \in M$ such that $\varrho(a, x) > 0$, $\varrho(x, b) > 0$, and $x$ saturates the triangle inequality: $\varrho(a, b) =\varrho(a, x) + \varrho(x,b)$. It was assumed in \cite{CK} that the graph is finite, but the proof works in the countable case as well.}

Now we recall the notion of a $p$-Wasserstein space in the special case when the underlying metric space is a countable graph metric space $(X,\varrho)$. The symbol $\mathcal{M}_+(X)$ stands for the set of nonnegative Borel measures on $X$. In this setting each measure $\mu\in\mathcal{M}_+(X)$ is uniquely determined by its value on singletons:
\begin{equation}\mu(A)=\sum_{x\in A}\mu(\{x\})\qquad\mbox{for all}\quad A\subseteq X,
\end{equation}
and therefore $\mu$ can be handled as a one-variable function on $X$. Such a function is often referred to as a \emph{probability mass function}. For the sake of simplicity, we will write shortly $\mu(x)$ instead of $\mu(\{x\})$. For a given real number $p\geq1$ we denote by $\mathcal{W}_p(X)$ the set of all probability measures such that
\begin{equation}
    \sum\limits_{x\in X}\varrho^p(x,\hat{x})\mu(x)<\infty
    \end{equation}
for some (hence all) $\hat{x}\in X$. The \emph{support} ${\mu}$ of a $\mu\in\mathcal{W}_p(X)$ in this setting equals to the set $\{x\in X\,|\,\mu(x)\neq0\}$.
A Borel probability measure $\pi$ on $X \times X$ is said to be a \emph{coupling} for $\mu$ and $\nu$ if the marginals of $\pi$ are $\mu$ and $\nu$, that is,
\begin{equation}
\sum\limits_{x'\in X}\pi(x,x')=\mu(x)\qquad\mbox{and}\qquad\sum\limits_{x\in X}\pi(x,x')=\nu(x').
\end{equation}
The set of all couplings (which is never empty because the product measure is a coupling) is denoted by $\Pi(\mu,\nu)$.  We will refer to couplings as transport plans, as $\pi(x,x')$ is the weight of mass that is transported from $x$ to $x'$ while $\mu$ is transported to $\nu$ along $\pi$. For a given measure $\mu$ we will denote by $\pi_{\mu}\in\Pi(\mu,\mu)$ the coupling which leaves $\mu$ undisturbed, that is, $\pi_{\mu}(x,x):=\mu(x)$ for all $x\in X$ and $\pi(x,y):=0$ otherwise.

If the cost function on $X\times X$ is $\varrho^p$, then the optimal cost of transporting $\mu$ into $\nu$ is
\begin{equation}\label{def:dw}
    d_p(\mu,\nu):=\Big(\inf_{\pi\in\Pi(\mu,\nu)}\sum_{(x,x')\in X\times X}\varrho^p(x,x')\cdot\pi(x,x')\Big)^{1/p}.
\end{equation}
It is known (see e.g. Theorem 1.5 in \cite{AG} with $c=\varrho^p$) that the infimum in \eqref{def:dw} is in fact a minimum. Those transport plans that minimise the transport cost are called \emph{optimal transport plans}.
We will refer to the metric space $\big(\mathcal{W}_p(X),d_p\big)$ as the $p$-Wasserstein space $\mathcal{W}_p(X)$. Let us denote the set of all \emph{finitely supported probability measures} by $\mathcal{F}(X)$. A very important feature of $p$-Wasserstein spaces is that if $p\geq1$, then $X$ embeds into $\mathcal{W}_p(X)$ isometrically (that is, $d_p(\delta_x,\delta_y)=\varrho(x,y)$ for all $x,y\in X$) and that $\mathcal{F}(X)$ is dense in $\mathcal{W}_p(X)$ (see e.g. Example 6.3 and Theorem 6.16 in \cite{Villani}). Although it is known that the isometry group of $X$ embeds into the isometry group of $\Wp(X)$, we provide with a short proof here for the sake of completeness.
\begin{proposition}\label{embedding}
Let $(X,\varrho)$ be countable graph metric space and let $p\geq1$ be fixed. Then the push-forward $\psi\mapsto\psi_\#$ defined by
\begin{equation}
\big(\psi_\#(\mu)\big)(x):=\mu\big(\psi^{-1}(x)\big)\qquad(x\in X)
\end{equation}
induces an embedding, which is in fact a group homomorphism
\begin{equation}
    \#:\mathrm{Isom}(X,\varrho)\to\mathrm{Isom}(\Wp(X),d_p).
\end{equation}
\end{proposition}
\begin{proof} Let us fix a $\psi\in\mathrm{Isom}(X,\varrho)$, two measures $\mu,\nu\in\Wp(X)$, and an optimal transport plan $\widetilde{\pi}\in\Pi(\mu,\nu)$. Since $\mathrm{Isom}(X,\varrho)$ is a group, $\psi^{-1}$ is a bijection such that $\varrho(\psi(x),\psi(y))=\varrho(x,y)=\varrho(\psi^{-1}(x),\psi^{-1}(y))$ for all $x,y\in X$. Furthermore, $\pi'(x,y):=\widetilde{\pi}(\psi^{-1}(x),\psi^{-1}(y))$ is a coupling for $\psi_\#(\mu)$ and $\psi_\#(\nu)$, because
$\sum\limits_{y\in X}\widetilde{\pi}(\psi^{-1}(x),\psi^{-1}(y))=\mu(\psi^{-1}(x))=\psi_\#(\mu)(x)$ and
$\sum\limits_{x\in X}\widetilde{\pi}(\psi^{-1}(x),\psi^{-1}(y))=\nu(\psi^{-1}(y))=\psi_\#(\nu)(y)$. Using that the $p$-Wasserstein distance is always smaller or equal to the cost of any coupling, the above observation implies
\begin{equation}
\begin{split}
    d_p^p(\mu,\nu)&=\sum_{(x,y)\in X\times X}\varrho^p(x,y)\cdot\widetilde{\pi}(x,y)\\
    &=\sum_{(x,y)\in X\times X}\varrho^p(\psi^{-1}(x),\psi^{-1}(y))\cdot\widetilde{\pi}(\psi^{-1}(x),\psi^{-1}(y))\\
    &=\sum_{(x,y)\in X\times X}\varrho^p(x,y)\cdot\widetilde{\pi}(\psi^{-1}(x),\psi^{-1}(y))\geq d_p^p(\psi_\#(\mu),\psi_\#(\nu)).
    \end{split}
\end{equation}
The reverse inequality $d_p^p(\psi_\#(\mu),\psi_\#(\nu))\geq d_p^p(\mu,\nu)$ can be proved along the same lines by using an optimal coupling $\widehat{\pi}\in\Pi(\psi_\#(\mu),\psi_\#(\nu))$ and the observation that $\pi''$ defined by $\pi''(x,y):=\widehat{\pi}(\psi(x),\psi(y))$ is a coupling for $\mu$ and $\nu$.

Finally, we verify that the map $\psi\mapsto\psi_\#$ is indeed a group homomorphism. For all $\psi,\chi\in\mathrm{Isom}(X,\varrho)$ and for all $\mu\in\Wp(X)$ and $x\in X$ we have
\begin{equation}
\begin{split}
\Big(\big(\psi\circ\chi\big)_\#(\mu)\Big)(x)&=\mu\Big(\big(\psi\circ\chi\big)^{-1}(x)\Big)=\mu\Big(\chi^{-1}\big(\psi^{-1}(x)\big)\Big)\\
&=\Big(\chi_\#(\mu)\Big)\big(\psi^{-1}(x)\big)=\Big(\psi_\#\big(\chi_\#(\mu)\big)\Big)(x).
\end{split}
\end{equation}
\end{proof}
Isometries of the form $\psi_\#$ are called \emph{trivial isometries}. We say that $\mathcal{W}_p(X)$ is \emph{isometrically rigid} if the map $\#$ is onto, i.e. $\mathrm{Isom}(X,\varrho)\cong\mathrm{Isom}(\Wp(X),d_p)$. In other words, if every isometry of $\Wp(X)$ is trivial.\\

In order to prove isometric rigidity, it would be useful to find properties which can be characterized by means of the metric, and thus are preserved by isometries. As we will see later in Proposition \ref{p:propAB}, the neighbouring property -- which says that two probability mass functions differ at the end-points of a given edge and nowhere else  (see Figure \ref{fig:n} below) -- is one of such properties.
\begin{definition} For a given $\alpha\in(0,1]$ we say that two measures $\mu,\nu\in\Wp(X)$ are \emph{$\alpha$-neighbouring}, if there exists an $\eta\in\mathcal{M}_+(X)$ and $u,v\in X$ with $\varrho(u,v)=1$ such that $\mu=\eta+\alpha\delta_u$ and $\nu=\eta+\alpha\delta_v$. In symbols, we write $\mu\equiv\nu~[\alpha]$.
\end{definition}
  	\begin{figure}[H]
	\centering
	\includegraphics[scale = 0.5]{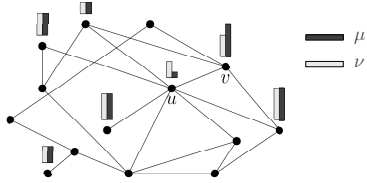}
\caption{An illustration of $\alpha$-neighbouring property.}
\label{fig:n}
	\end{figure}
In particular, $\mu\equiv\nu~[1]$ means that $\mu=\delta_u$ and $\nu=\delta_v$, where $u$ and $v$ are different endpoints of an edge in the underlying graph. We  note that $\equiv$ is not an equivalence relation, as it is not reflexive and not transitive. In order to see the metric side of the neighbouring property, let us introduce the set $B_s(\mu,\nu)$ for $\mu,\nu\in\Wp(X)$
\begin{equation}
    B_s(\mu, \nu):=\Big\{\xi\in\mathcal{W}_p(X)\,\Big|\,d_{p}(\mu, \xi)\le \sqrt[p]{s}d_p(\mu,\nu), \ d_{p}(\xi,\nu)\le\sqrt[p]{(1-s)}d_p(\mu,\nu) \Big\}.
\end{equation}
The following proposition says that $B_s(\mu,\nu)$ is always non-empty. For the sake of brevity, we will use the notation in the sequel $\xi_s^{\mu,\nu}:=(1-s)\mu+ s\nu$.
\begin{proposition}\label{uj} Let $(X,\varrho)$ be a countable graph metric space and let $p\geq1$ and $s\in(0,1)$ be fixed real numbers. Then for any $\mu,\nu\in\Wp(X)$ the measure $\xi_s^{\mu,\nu}$ belongs to $B_s(\mu,\nu)$. In particular, the statement that $B_s(\mu,\nu)$ is a singleton is equivalent to $B_s(\mu,\nu)=\{\xi_s^{\mu,\nu}\}$.
\end{proposition}
\begin{proof}
First we show that $\xi_s^{\mu,\nu}$ satisfies $d_p(\mu,\xi_s^{\mu,\nu})\leq\sqrt[p]{s}d_p(\mu,\nu)$. Let us fix a optimal transport plan $\widetilde{\pi}\in\Pi(\mu,\nu)$. Recall that $\pi_{\mu}$ is a coupling which leaves $\mu$ undisturbed: $\pi_{\mu}(x,x)=\mu(x)$ and $\pi_{\mu}(x,y)=0$ otherwise. Then $\pi_s:=(1-s)\pi_{\mu}+s\widetilde{\pi}\in\Pi(\mu,\xi_s^{\mu,\nu})$. Indeed, $\sum\limits_{y\in X}\Big((1-s)\pi_{\mu}(x,y)+s\widetilde{\pi}(x,y)\Big)=(1-s)\mu(x)+s\mu(x)=\mu(x)$ for all $x\in X$ and $\sum\limits_{x\in X}\Big((1-s)\pi_{\mu}(x,y)+s\widetilde{\pi}(x,y)\Big)=(1-s)\mu(y)+s\nu(y)=\xi_s^{\mu,\nu}(y)$ for all $y\in X$.
Using the transport plan $\pi_s\in\Pi(\mu,\xi_s)$ we can estimate $d_p(\mu,\xi_s^{\mu,\nu})$ as
\begin{equation}
\begin{split}
    d_p^p(\mu,\xi_s^{\mu,\nu})&=\inf_{\pi\in\Pi(\mu,\xi_s)}\sum_{(x,y)\in X\times X}\varrho^p(x,y)\cdot\pi(x,y)\leq\sum_{(x,y)\in X\times X}\varrho^p(x,y)\cdot\pi_s(x,y)\\
    &=(1-s)\sum_{(x,y)\in X\times X}\varrho^p(x,y)\cdot\pi_{\mu}(x,y)+s\sum_{(x,y)\in X\times X}\varrho^p(x,y)\cdot\widetilde{\pi}(x,y)\\
    &=s d_p^p(\mu,\nu),
\end{split}
\end{equation}
where we used $\widetilde{\pi}$ is optimal and that $\varrho^p(x,y)\cdot\pi_{\mu}(x,y)=0$ for all $(x,y)\in X\times X$. The other inequality $d_p(\xi_s^{\mu,\nu},\nu)\leq\sqrt[p]{1-s}d_p(\mu,\nu)$ can be proved in the same way using a similar combination of $\widetilde{\pi}$ and $\pi_{\nu}$ (which leaves $\nu$ undisturbed).
\end{proof}

In Proposition \ref{uj} we saw that $\xi_s^{\mu,\nu}\in B_s(\mu,\nu)$ for all $\mu,\nu\in\Wp(X)$. Our next aim is to find a metric characterization for those pairs such that $B_s(\mu,\nu)=\{\xi_s^{\mu,\nu}\}$.
\begin{proposition}\label{p:propAB} Let $(X,\varrho)$ be a countable graph metric space and let $p\geq1$ and $\alpha\in(0,1]$ be fixed real numbers. Then the following statements are equivalent:
\begin{itemize}
    \item[(i)] $\mu$ and $\nu$ are $\alpha$-neighbouring, that is, there exists an $\eta\in\mathcal{M}_+(X)$ and $u,v\in X$ with $\varrho(u,v)=1$ such that
    \begin{equation}\label{e:uv}
\mu=\eta+\alpha\delta_u\quad\mbox{and}\quad \nu=\eta+\alpha\delta_v.
\end{equation}
    \item[(ii)]
         $d_{p}(\mu, \nu)=\sqrt[p]{\alpha}$ and $B_s(\mu,\nu)=\{\xi_s^{\mu,\nu}\}$ for all $s\in (0,1).$
    \item[(iii)]
         $d_{p}(\mu, \nu)=\sqrt[p]{\alpha}$ and $B_{\frac{1}{2}}(\mu,\nu)=\{\xi_{\frac{1}{2}}^{\mu,\nu}\}$.
\end{itemize}
\end{proposition}
\begin{proof}
\emph{(i)$\Longrightarrow$(ii):} First we show that $d_p(\mu,\nu)\geq\sqrt[p]{|\nu(\hat{x})-\mu(\hat{x})|}$ holds for all  $\mu,\nu\in\Wp(X)$ and $\hat{x}\in X$. By symmetry, we can assume without loss of generality that $\mu(\hat{x})\leq\nu(\hat{x})$. Then for any coupling $\pi\in\Pi(\mu,\nu)$ we have
\begin{equation}
\begin{split}
\nu(\hat{x})&=\sum_{x\in X}\pi(x,\hat{x})=\pi(\hat{x},\hat{x})+\sum_{\substack{x\in X \\ x\neq\hat{x}}}\pi(x,\hat{x})\\
&\leq\sum_{y\in X}\pi(\hat{x},y)+\sum_{\substack{x\in X \\ x\neq\hat{x}}}\pi(x,\hat{x})=\mu(\hat{x})+\sum_{\substack{x\in X \\ x\neq\hat{x}}}\pi(x,\hat{x}).
\end{split}
\end{equation}
which implies
\begin{equation}\nu(\hat{x})-\mu(\hat{x})\leq\sum_{\substack{x\in X \\ x\neq\hat{x}}}\pi(x,\hat{x}).
\end{equation}
Since $x\neq y$ implies $\varrho^p(x,y)\geq1$, we have the following lower bound for the cost of $\pi$
\begin{equation}
    \nu(\hat{x})-\mu(\hat{x})\leq\sum_{\substack{x\in X \\ x\neq\hat{x}}}\pi(x,\hat{x})\leq\sum_{x\in X}\varrho^p(x,\hat{x})\cdot\pi(x,\hat{x})\leq\sum_{(x,y)\in X\times X}\varrho^p(x,y)\cdot\pi(x,y).
\end{equation}
By taking the minimum over $\Pi(\mu,\nu)$, one gets $\nu(\hat{x})-\mu(\hat{x})\leq d_p^p(\mu,\nu).$ In \eqref{e:uv} we have $\nu(v)-\mu(v)=\alpha$, and therefore $d_p(\mu,\nu)\geq\sqrt[p]{\alpha}$ holds. To see the reverse inequality, observe first that the following $\pi\in\Pi(\mu,\nu)$ has cost $\alpha$
\[
\pi(x,y)= \left\{
\begin{array}{ll}
      \mu(x) & \mbox{if}~~x=y~~\mbox{and}~~x\neq u, \\
      \mu(x)-\alpha &\mbox{if}~~x=y=u,\\
      \alpha &\mbox{if}~~x=u~~\mbox{and}~~y=v,\\
      0 & \mbox{otherwise.}\\
\end{array}
\right.
\]
And therefore, $d_p(\mu,\nu)\leq\sqrt[p]{\alpha}$. Now assume that $\xi\in B_s(\mu,\nu)$. We have to show that $\xi=\xi_{s}^{\delta_u,\delta_v}=\eta+(1-s)\alpha\delta_u+s\alpha\delta_v$. Let $\pi^*\in\Pi(\mu,\xi)$ be an optimal transport plan, i.e.
\begin{equation}\label{ximu}
\alpha s\geq d_p^p(\mu,\xi)=\sum_{(x,y)\in X\times X}\varrho^p(x,y)\cdot\pi^*(x,y).
\end{equation}
By adding $\pi^*(u,u)-\pi^*(u,u)=0$, the right hand side in \eqref{ximu} can be written as
\begin{equation}
\sum_{\substack{(x,y)\in X\times X\\ x\neq u}}\varrho^p(x,y)\cdot\pi^*(x,y)+\left[\sum_{y\in X}\varrho^p(u,y)\cdot\pi^*(u,y)+\pi^*(u,u)\right]-\pi^*(u,u),
\end{equation}
which gives
\begin{equation}\label{i:1}
\alpha s\geq\mu(u)-\pi^*(u,u),
\end{equation}
because
\begin{equation}
\sum_{\substack{(x,y)\in X\times X\\ x\neq u}}\varrho^p(x,y)\cdot\pi^*(x,y)\geq0,
\end{equation}
and if $u\neq y$, then $\varrho(u,y)\geq1$, and thus
\begin{equation}
\sum_{y\in X}\varrho^p(u,y)\cdot\pi^*(u,y)+\pi^*(u,u)\geq\sum_{y\in X}\pi^*(u,y)=\mu(u).
\end{equation}
Since $\mu(u)=\eta(u)+\alpha$ according to  \eqref{e:uv}, we can rearrange \eqref{i:1} as
\begin{equation}
\pi^*(u,u)\geq\alpha(1-s)+\eta(u).\end{equation}
Using that $\pi^*\in\Pi(\mu,\xi)$ we get $\xi(u)=\sum_{x\in X}\pi^*(x,u)\geq\pi^*(u,u)$, and thus
\begin{equation}\label{xiu}
\xi(u)\geq\pi^*(u,u)\geq\eta(u)+(1-s)\alpha\delta_u(u).
\end{equation}
Combination of \eqref{ximu} and \eqref{xiu} asserts now that
\begin{equation}\label{aa}
s\alpha\geq d_p^p(\mu,\xi)\geq\big|\xi(u)-\mu(u)|\geq|\eta(u)+(1-s)\alpha-\big(\eta(u)+\alpha\delta_u(u)\big)\big|=s\alpha.
\end{equation}
Furthermore, a very similar calculation with $\xi$ and $\nu$ gives
\begin{equation}\label{bb}
(1-s)\alpha\geq d_p^p(\nu,\xi)\geq|\xi(v)-\nu(v)|\geq\big|\eta(v)+s\alpha-\big(\eta(v)+\alpha\delta_v(v)\big)\big|=(1-s)\alpha.
\end{equation}
Since every inequality in \eqref{aa} and \eqref{bb} is actually an equality, we get
\begin{equation}
\xi(u)=\eta(u)+(1-s)\alpha\delta_u(u)\qquad\mbox{and}\qquad\xi(v)=\eta(v)+s\alpha\delta_v(v)
\end{equation}
which together with $d_p^p(\mu,\xi)=|\mu(u)-\xi(u)|$ and $d_p^p(\xi,\nu)=|\xi(v)-\nu(v)|$ imply that \begin{equation}
    \xi=\eta+(1-s)\alpha\delta_u+s\alpha\delta_v=\xi_{s}^{\delta_u,\delta_v}.
    \end{equation}
This proves that (i)$\Longrightarrow$ (ii).\\

\noindent\emph{(ii)$\Longrightarrow$(iii):} This implication is straightforward.\\

\noindent\emph{(iii)$\Longrightarrow$(i):} We have to show that if $d_p(\mu,\nu)=\sqrt[p]{\alpha}$ and $B_{\frac{1}{2}}(\mu,\nu)=\{\xi_{\frac{1}{2}}^{\mu,\nu}\}$, then \eqref{e:uv} holds. Let $\pi\in\Pi(\mu,\nu)$ be an optimal transport plan
\begin{equation}\label{BtoA}
\alpha=d_p^p(\mu,\nu)=\sum_{\substack{(x,y)\in\supp\pi\\ x\neq y}}\varrho^p(x,y)\cdot\pi(x,y).
\end{equation}
and recall that $\pi_{\frac{1}{2}}:=\frac{1}{2}\pi_{\mu}+\frac{1}{2}\pi\in \Pi(\mu,\xi_{\frac{1}{2}}^{\mu,\nu})$.
First assume indirectly that there exists an $(x',y')\in\supp\pi$ for which $k:=\varrho(x',y')>1$. Let us choose a path of length $k$ between $x'$ and $y'$ along the vertices $x_0=x', x_1,\dots,x_{k_1},x_k=y'$. Set $c=\frac{\pi(x',y')}{4}$ and modify $\xi_{\frac{1}{2}}^{\mu,\nu}$ along this path as
\begin{equation}
\xi:=\xi_{\frac{1}{2}}^{\mu,\nu}-c\delta_{x_0}+c\delta_{x_1}+c\delta_{x_{k-1}}-c\delta_{x_k}.
\end{equation}
  	\begin{figure}[H]
	\centering
	\includegraphics[scale = 0.6]{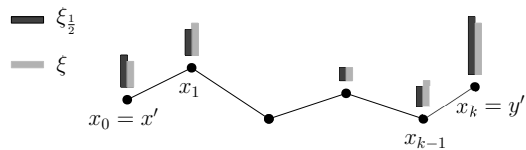}
	\caption{An illustration of how to modify $\xi_{\frac{1}{2}}^{\mu,\nu}$ in order to get $\xi$.}
	\end{figure}
Since $k>1$ and $p\geq1$, we have $\xi\neq\xi_{\frac{1}{2}}^{\mu,\nu}$ and it follows from the construction that
\begin{equation}
d_p^p(\mu,\xi)\leq d_p^p(\mu,\xi_{\frac{1}{2}}^{\mu,\nu})-k^pc+(k-1)^pc+1^pc\leq d_p^p(\mu,\xi_{\frac{1}{2}}^{\mu,\nu})\leq\frac{\alpha}{2}.
\end{equation}
A similar calculation shows that $d_p^p(\xi,\nu)\leq\frac{\alpha}{2}$, and thus $\xi\in B_{\frac{1}{2}}(\mu,\nu)$, a contradiction.
This contradiction means that $\varrho(x,y)=1$ for all $(x,y)\in\supp\pi$ in \eqref{BtoA}.

Now assume indirectly that $\supp\pi$ has at least two different elements $(x_1,y_1)$ and $(x_2,y_2)$. Set $\widetilde{c}:=\frac{\mathrm{min}\{\pi(x_1,y_1),\pi(x_2,y_2)\}}{4}$ and modify $\xi_{\frac{1}{2}}^{\mu,\nu}$ as follows
\begin{equation}
\xi:=\xi_{\frac{1}{2}}^{\mu,\nu}+\widetilde{c}\delta_{x_1}+\widetilde{c}\delta_{y_1}-\widetilde{c}\delta_{x_2}-\widetilde{c}\delta_{y_2}.
\end{equation}
Again, we have that $\xi\neq\xi_{\frac{1}{2}}^{\mu,\nu}$. In order to give an upper bound for $d_p^p(\mu,\xi)$, let us define the coupling $\widetilde{\pi}\in \Pi(\mu,\xi)$ by modifying $\pi_{\frac{1}{2}}=\frac{1}{2}\pi_{\mu}+\frac{1}{2}\pi\in \Pi(\mu,\xi_{\frac{1}{2}}^{\mu,\nu})$ as follows
\[
\widetilde{\pi}(x,y)= \left\{
\begin{array}{ll}
      \pi_{\frac{1}{2}}(x_1,y_1)+\widetilde{c} & \mbox{if}~(x,y)=(x_1,y_1),\\
      \pi_{\frac{1}{2}}(x_2,y_2)-\widetilde{c} & \mbox{if}~(x,y)=(x_2,y_2),\\
      \pi_{\frac{1}{2}}(x,y) & \mbox{otherwise}. \\
\end{array}
\right.
\]
Since $\varrho(x,y)=1$ for all $(x,y)\in\supp{\pi}$ and $\pi_{\mu}(x,y)\varrho(x,y)=0$ for all $x,y\in X$, we have
\begin{equation}
d_p^p(\mu,\xi)\leq\sum_{\substack{(x,y)\in\supp{\widetilde{\pi}}\\ x\neq y}}\widetilde{\pi}(x,y)=\sum_{\substack{(x,y)\in\supp{\pi}\\ x\neq y}}\frac{\pi(x,y)}{2}=\frac{\alpha}{2}.
\end{equation}
Similarly, $d_p^p(\xi,\nu)\leq\frac{\alpha}{2}$, and thus $\xi$ belongs to $B_{\frac{1}{2}}(\mu,\nu)$, a contradiction. The only remaining possibility is that there exists $u,v\in X$ such that $\varrho(u,v)=1$ and \eqref{BtoA} can be written as $\alpha=d_p^p(\mu,\nu)=\varrho^p(u,v)\cdot\pi(u,v)=\pi(u,v)$, which means exactly that \eqref{e:uv} holds. This proves (i)$\Longleftrightarrow$(ii).
\end{proof}

\begin{corollary}\label{c:n-preserver} Let $(X,\varrho)$ be a countable graph metric space and let $p\geq1$ and $\alpha\in(0,1]$ be fixed real numbers. For any isometry $\Phi:\Wp(X)\to\Wp(X)$ and for any pair of measures $\mu,\nu\in\Wp(X)$ the following holds
\begin{equation}\label{f:n-preserver}
    \mu\equiv\nu~[\alpha]\qquad\Longleftrightarrow\qquad\Phi(\mu)\equiv\Phi(\nu)~[\alpha].
\end{equation}
\end{corollary}
\begin{proof}
Since $\Phi$ and $\Phi^{-1}$ are both distance preserving bijections, we have that $B_s(\mu,\nu)$ is a singleton if and only if $B_s\big(\Phi(\mu),\Phi(\nu)\big)$ is a singleton. According to Proposition \ref{p:propAB}, this implies that $\mu\equiv\nu~[\alpha]$ if and only if $\Phi(\mu)\equiv\Phi(\nu)~[\alpha]$.
\end{proof}

\section{The main result}

Now we are ready to state and prove the main result of the paper.

\begin{theorem}\label{t:main}
Let $(X,\varrho)$ be a countable graph metric space and let $p\geq1$ be fixed. Then the $p$-Wasserstein space $\Wp(X)$ is isometrically rigid, i.e., $\mathrm{Isom}(\Wp(X),d_p)\cong\mathrm{Isom}(X,\varrho)$.
\end{theorem}

\begin{proof}
We have seen in Proposition \ref{embedding} that $\#:\mathrm{Isom}(X,\varrho)\to\mathrm{Isom}(\Wp(X),d_p)$ is a group homomorphism. Therefore, it is enough to prove that it is surjective, i.e., for any $\Phi\in\mathrm{Isom}(\Wp(X),d_p)$ there exists a $\psi\in\mathrm{Isom}(X,\varrho)$ such that $\Phi=\psi_\#$. The strategy of proof is similar to the sketch (Step 1--3.) mentioned in the introduction.\\

\noindent\underline{Step 1.} First we prove that $\Phi$ maps the set of Dirac masses onto itself. Assume that $\mu=\delta_u$ is a Dirac measure and choose a $v\in X$ such that $\varrho(u,v)=1$. Since $\delta_u\equiv\delta_v~[1]$, we have $\Phi(\delta_u)\equiv\Phi(\delta_v)~[1]$ according to Corollary \ref{c:n-preserver}. That is, $\Phi(\mu)=\delta_{\hat{u}}$ and $\Phi(\nu)=\delta_{\hat{v}}$ for some $\hat{u},\hat{v}\in X$ with $\varrho(\hat{u},\hat{v})=1$. In particular, $\Phi(\mu)$ is a Dirac measure. Since $d_p(\delta_x,\delta_y)=\varrho(x,y)$ for all $x,y\in X$ and $\Phi^{-1}$ is an isometry as well, we see that $\Phi$ maps the set of Dirac masses bijectively onto itself, and the function $\psi:X\to X$ defined by \begin{equation}
\Phi(\delta_x):=\delta_{\psi(x)}\qquad(x\in X)
\end{equation}
is an isometry. Let us consider the isometry $\widetilde{\Phi}:=\psi^{-1}_{\#}\circ\Phi$. On the one hand, $\widetilde{\Phi}$ fixes all Dirac measures. On the other hand, if $\widetilde\Phi(\mu)=\mu$ for all $\mu\in\Wp(X)$, then $\Phi=\psi_\#$. Therefore we can assume without loss of generality that $\psi(x)=x$, and hence $\Phi(\delta_x)=\delta_x$ for all $x\in X$.\\

\noindent\underline{Step 2.} Our next task is to prove that $\Phi$ leaves a dense set of finitely supported measures fixed. We do this by induction with respect to the prescribed location of the supports in question. Since the underlying graph $G(X,E)$ is connected, starting with an arbitrary $x_1\in X$ we can enumerate $X$ in a way that for all $n\geq2$ the vertex $x_n$ is connected to the initial segment $X_{n-1}=\{x_1,\dots,x_{n-1}\}$ with at least one edge. If $\mu\in\mathcal{F}(X)$, then $\supp\mu\subseteq X_N$ for some large enough $N\in\mathbb{N}$, and therefore it is enough to show that measures supported in $X_n$ are fixed by $\Phi$ for all $n\in\mathbb{N}$.\\

If $n=1$, then $\supp{\mu}\subseteq X_1$ implies $\mu=\delta_{x_1}$, and $\Phi(\delta_{x_1})=\delta_{x_1}$ according to Step 1.\\

If $n=2$, then $\supp{\mu}\subseteq X_2$ implies $\mu=(1-s)\delta_{x_1}+s\delta_{x_2}$ for some $s\in[0,1]$. If $s=0$ or $s=1$ then $\mu$ is a Dirac measure and thus $\Phi(\mu)=\mu$. Assume now that $0<s<1$. In this case, $\mu=\xi_s^{\delta_{x_1},\delta_{x_2}}$ and we know from Proposition \ref{uj} that $\xi_s^{\delta_{x_1},\delta_{x_2}}\in B_s(\delta_{x_1},\delta_{x_2})$. Moreover, $\Phi(\xi_s^{\delta_{x_1},\delta_{x_2}})\in B_s(\delta_{x_1},\delta_{x_2})$ holds as well, because $\Phi(\delta_{x_i})=\delta_{x_i}$ for $i=1,2$, and thus
\begin{equation}
    d_p(\delta_{x_1},\Phi(\mu))=d_p(\Phi(\delta_{x_1}),\Phi(\mu))=d_p(\delta_{x_1},\mu)\leq\sqrt[p]{s}d_p(\delta_{x_1},\delta_{x_2}),
\end{equation}
and
\begin{equation}
    d_p(\Phi(\xi_s^{\delta_{x_1},\delta_{x_2}}),\delta_{x_2})=d_p(\Phi(\xi_s^{\delta_{x_1},\delta_{x_2}}),\Phi(\delta_{x_2}))=d_p(\xi_s^{\delta_{x_1},\delta_{x_2}},\delta_{x_2})\leq\sqrt[p]{1-s}d_p(\delta_{x_1},\delta_{x_2}).
\end{equation}
But $\delta_{x_1}\equiv\delta_{x_2}~[1]$, and thus $B_s(\delta_{x_1},\delta_{x_2})$ is a singleton according to Proposition \ref{p:propAB}. This implies that $\Phi(\xi_s^{\delta_{x_1},\delta_{x_2}})=\xi_s^{\delta_{x_1},\delta_{x_2}}$.\\

Assume now that $\Phi(\zeta)=\zeta$ holds whenever $\supp\zeta\subseteq X_n=\{x_1,\dots,x_n\}$, and choose a finitely supported measure $\mu=\sum_{i=1}^{n+1}\mu(x_i)\delta_{x_i}$ which satisfies the following two properties:
\begin{equation}\label{p1}
\mbox{for all}~~ x,y\in\supp\mu: \quad x\neq y~~\mbox{implies}~~\mu(x)\neq\mu(y),
\end{equation}
and
\begin{equation}\label{p2}
\mbox{for all pairwise different elements}~~ x,y,z\in\supp\mu:\quad \mu(x)+\mu(y)\neq\mu(z).
\end{equation}
If $\mu(x_{n+1})=0$, then $\supp\mu\subseteq X_n$, and if $\mu(x_{n+1})=1$, then $\mu=\delta_{x_{n+1}}$. In both cases, $\Phi(\mu)=\mu$ according to the inductive hypothesis and Step 1. So it remains to deal with the case $0<\mu(x_{n+1})<1$.
 According to the construction, there exists an $u\in X_n$ such that $\varrho(u, x_{n+1})=1$. Set $c:=\mu(u)+\mu(x_{n+1})$ and define two (in a sense extremal) measures $\mu_*$ and $\mu^*$ as follows
\begin{equation}\mu_*:=\mu-\mu(u)\delta_{u}+\mu(u)\delta_{x_{n+1}}\quad\mbox{and}\quad \mu^*:=\mu-\mu(x_{n+1})\delta_{x_{n+1}}+\mu(x_{n+1})\delta_{u}.
\end{equation}
Observe that $\mu_*\equiv\mu^*~[c]$, and $\supp{\mu^*}\subseteq X_n$ implies $\Phi(\mu^*)=\mu^*$. Furthermore, we have
\begin{equation}\label{muprop}\mu_*(u)=\mu^*(x_{n+1})=0,\qquad\mu_*(x_{n+1})=\mu^*(u)=c
\end{equation}
and
\begin{equation}\label{nb}
    \mu_*\equiv\mu~[\mu(u)],\qquad\mu^*\equiv\mu~[\mu(x_{n+1})].
\end{equation}
Let us define a curve $\gamma$ which connects $\mu^*$ and $\mu_*$ (see Figure \ref{teleport} below)
\begin{equation}\label{gamma}
\gamma:[0,c]\to\mathcal{W}_p(X);\qquad\gamma(t):=\mu^*+t\delta_{x_{n+1}}-t\delta_u.
\end{equation}
\begin{figure}[H]
\centering
\includegraphics[scale = 0.4]{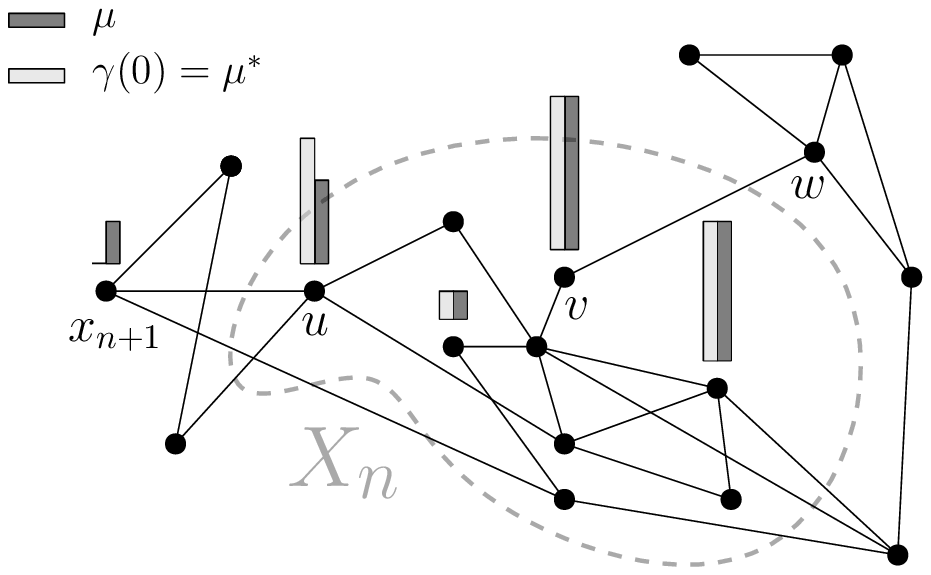}\quad
\includegraphics[scale = 0.4]{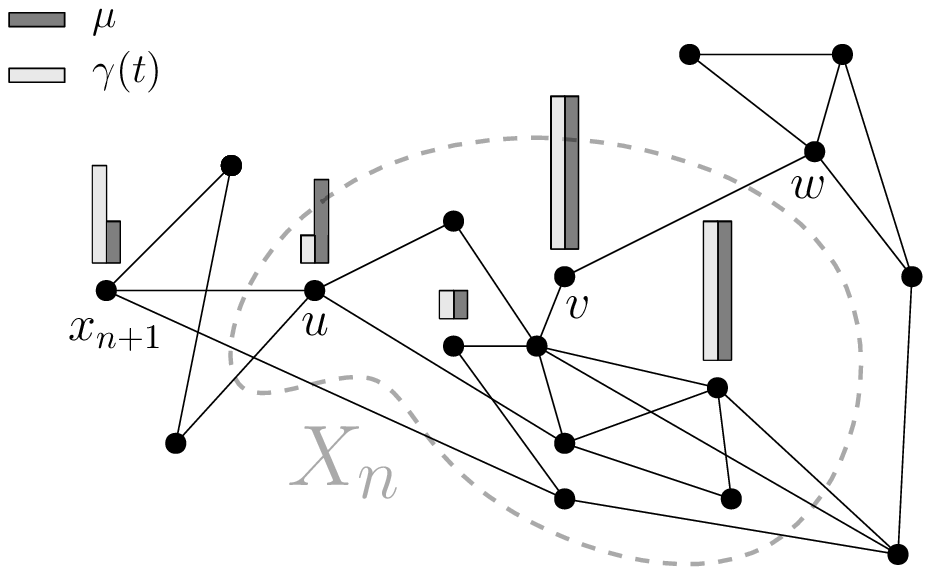}\quad
\includegraphics[scale = 0.4]{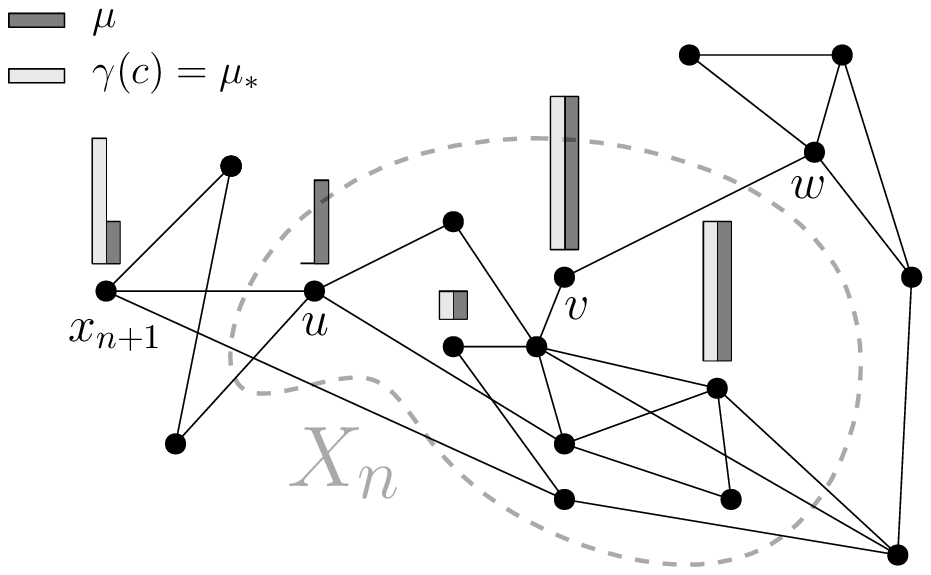}
\caption{An illustration of teleporting $\mu^*$ to $\mu_*$ along the curve $\gamma$.}
\label{teleport}
\end{figure}
Observe that $\gamma(0)=\mu^*$, $\gamma\big(\mu(x_{n+1})\big)=\mu$, $\gamma(c)=\mu_*$, and that
\begin{equation}\label{gteq}
\gamma(t)\equiv\mu~[|t-\mu(x_{n+1})|]\quad\mbox{for all}\quad t\in[0,c]\setminus\{x_{n+1}\}.
\end{equation}

According to Corollary \ref{c:n-preserver}, $\mu_*\equiv\mu^*~[c]$ implies $\Phi(\mu_*)\equiv\Phi(\mu^*)~[c]$, and thus we can issue a curve $\Gamma$ from $\Phi(\mu^*)$ to $\Phi(\mu_*)$ with a structure similar to that of $\gamma$. Recall that $\supp{\mu^*}\subseteq X_n$ implies $\Phi(\mu^*)=\mu^*$, and therefore $\Phi(\mu_*)\equiv\Phi(\mu^*)~[c]$ can be written as  $\Phi(\mu_*)\equiv\mu^*~[c]$. Proposition \ref{p:propAB} indicates that there exist a $v\in\supp{\mu^*}$ and a $w\in X$ with $\varrho(v,w)=1$ such that $\Phi(\mu_*)=\mu^*-c\delta_v+c\delta_w$. Therefore the natural way to connect $\mu^*=\Phi(\mu^*)$ and $\Phi(\mu_*)$ with a curve is
\begin{equation}\label{Gamma}
\Gamma:[0,c]\to\Wp(X);\qquad \Gamma(t):=\mu^*-t\delta_v+t\delta_w.
\end{equation}
It follows from \eqref{f:n-preserver} and \eqref{nb} that $\Phi(\mu)\equiv\Phi(\mu^*)~[\mu(x_{n+1})]$ and $\Phi(\mu)\equiv\Phi(\mu_*)~[\mu(u)]$. And thus, $\Phi(\mu)$ must have the form $\mu^*-t\delta_v+t\delta_w$ for some $t\in(0,c)$. In fact, $\Phi(\mu)=\Gamma\big(\mu(x_{n+1}))$. And similarly, it follows form \eqref{gteq} and Corollary \ref{c:n-preserver} that the $\Phi$-image of $\gamma$ is $\Gamma$.

We need more, in fact we want to prove that $\gamma=\Gamma$, or equivalently, $x_{n+1}=w$ and $u=v$. To show that $u=v$, it is enough to prove that $\mu^*(u)=\mu^*(v)$. Indeed, assume indirectly that $\mu^*(u)=\mu^*(v)$, but $u\neq v$. First recall that $u\neq v$ implies $\mu^*(v)=\mu(v)$. If $\mu(u)=0$, then
\begin{equation}
    0<\mu(x_{n+1})=\mu(x_{n+1})+\mu(u)=\mu^*(u)=\mu^*(v)=\mu(v)
\end{equation}
which contradicts \eqref{p1}, unless $v=x_{n+1}$. But $v=x_{n+1}$ is impossible because $v\in X_n$ and $x_{n+1}\notin X_n$. If $\mu(u)\neq0$, then $\mu(v)=\mu^*(v)=\mu^*(u)>\mu(u)>0$, and thus $u,v,x_{n+1}$ are pairwise different elements of $\supp{\mu}$ such that \begin{equation}\mu(v)=\mu^*(v)=\mu^*(u)=\mu(u)+\mu(x_{n+1}),
\end{equation}
which contradicts \eqref{p2}. Now we know that if $\mu^*(u)=\mu^*(v)$, then $u=v$. In this case, $\Phi(\mu)$ can be written as $\Phi(\mu)=\mu^*-\mu(x_{n+1})\delta_u+\mu(x_{n+1})\delta_w$, and thus
\begin{equation}\label{1a}
    d_p^p(\delta_{x_{n+1}},\Phi(\mu))=d_p^p(\delta_{x_{n+1}},\mu^*)-\mu(x_{n+1})+\varrho^p(x_{n+1},w)\mu(x_{n+1}).
\end{equation}
Moreover, we have that $d_p(\delta_{x_{n+1}},\Phi(\mu))=d_p(\Phi(\delta_{x_{n+1}}),\Phi(\mu))=d_p(\delta_{x_{n+1}},\mu)$, and that
\begin{equation}\label{1b}
d_p^p(\delta_{x_{n+1}},\mu)=d_p^p(\delta_{x_{n+1}},\mu^*)-\mu(x_{n+1}).
\end{equation}
Now we can conclude from \eqref{1a} and \eqref{1b} that $\varrho^p(x_{n+1},w)\mu(x_{n+1})=0$. Since $\mu(x_{n+1})\neq0$, we get $x_{n+1}=w$.

What remains to prove in this step is that $c=\mu^*(u)=\mu^*(v)$. On the one hand, we know from \eqref{muprop} that $\mu^*(u)=c$. On the other hand, we have $\mu^*(v)\geq c$, because
\begin{equation}
    0\leq\Phi(\mu^*)(v)=\mu^*(v)-c\delta_v(v)+c\delta_w(v)=\mu^*(v)-c.
\end{equation}
Assume indirectly that $\mu^*(v)>c$. In this case, we can extend $\Gamma$ from $[0,c]$ to $[0,\mu^*(v)]$ by
\begin{equation}\label{gammahullam}
\widetilde{\Gamma}:[0,\mu^*(v)]\to\Wp(X);\qquad\widetilde{\Gamma}(t):=\mu^*-t\delta_v+t\delta_w.
\end{equation}
This extension has the property that
\begin{equation}\Phi(\mu)\equiv\widetilde{\Gamma}(t)~[|t-\mu(x_{n+1})|]\quad\mbox{for all }\quad t\in[0,\mu^*(v)]\setminus\{\mu(x_{n+1})\}.
\end{equation}
Since $\Phi^{-1}$ is an isometry as well, Corollary \ref{c:n-preserver} says that $\Phi^{-1}\big(\widetilde{\Gamma}(t)\big)\equiv\mu~[|t-\mu(x_{n+1}|]$ holds for all $t\in[0,\mu^*(v)]\setminus\{\mu(x_{n+1})\}$, and thus $\gamma$ can be extended through $\gamma(c)=\mu_*$ with measures which are all in neighbouring relation with $\mu$, a contradiction. Indeed, in order to continue $\gamma$, we need to add more weight to $x_{n+1}$. But $\mu_*(u)=0$, so we should teleport mass from a point $x\in X\setminus\{x_{n+1},u\}$ which would ruin the neighbouring relation with $\mu$. Summarising the above observations: $\mu^*(u)=\mu^*(v)$, and thus $\Phi(\mu)=\mu$.\\

\noindent\underline{Step 3.} We saw that $\Phi(\mu)=\mu$ holds if $\mu$ satisfies \eqref{p1} and \eqref{p2}. Since $\Phi$ is continuous and $\mathcal{F}(X)$ is dense in $\Wp(X)$, it is enough to show that every $\nu:=\sum_{i=1}^L a_i\delta_{u_i}\in\mathcal{F}(X)$ can be approximated by such measures. Here we assume that $u_i\neq u_j$ if $i\neq j$ and that $a_i>0$ for all $1\leq i\leq L$. For an arbitrary $\varepsilon>0$ we are going to construct a measure $\nu'$ with $\supp{\nu}=\supp{\nu'}$  such that $d_p(\nu,\widetilde{\nu})<\varepsilon$.

If $L=1$ then $\nu$ itself satisfies \eqref{p1} and \eqref{p2}, so we can assume that $L\geq2$. Set
\begin{equation}K:=\max\{\varrho(u_i,u_j)\,|\,1\leq i,j\leq L\},
\end{equation}
and if necessary, choose a smaller $0<\widetilde\varepsilon\leq\varepsilon$ such that $0<a_i-\frac{\widetilde\varepsilon}{K^pL^2}$ holds for all $1\leq i\leq L$. Using such an $\widetilde\varepsilon$, the intersection of the cube
\begin{equation}
    C=\prod_{i=1}^n\left[a_i-\frac{\widetilde\varepsilon^p}{K^pL^2},a_i+\frac{\widetilde\varepsilon^p}{K^pL^2}\right]\subseteq\mathbb{R}^L
\end{equation}
 with the hyperplane $P=\{(c_1,\dots,c_L)\,|\,\sum_{i=1}^Lc_i=1\}$ contains only vectors $(c_1,\dots,c_L)$ such that $\sum_{i=1}^Lc_i\delta_{u_i}\in\Wp(X)$. The set of representing vectors of those measures which violate  \eqref{p1} or \eqref{p2} can be covered by the union of finitely many lower dimensional linear subspaces in $\mathbb{R}^L$. Since none of these subspaces are identical with $P$, we can choose an uncovered $(c_1,\dots,c_L)\in C\cap P$ and set $\nu'=\sum_{i=1}^Lc_i\delta_{u_i}$.

 We claim that $d_p(\nu,\nu')<\varepsilon$. To see this, we construct a $\pi\in\Pi(\nu,\nu')$ which leaves all mass shared by $\nu$ and $\nu'$ undisturbed. Set $m_i=\min\{a_i,c_i\}$ $(1\leq i \leq L)$ and $M:=\sum_{i=1}^L m_i$, and subtract $\sum_{i=1}^L m_i\delta_{u_i}$ from $\nu$ and $\nu'$. Now we have $\zeta:=\sum_{i=1}^L(a_i-m_i)\delta_{u_i}$ and $\zeta':=\sum_{i=1}^L(c_i-m_i)\delta_{u_i}$ with $\supp{\zeta}\cap\supp{\zeta'}=\emptyset$. For the product measure $\big(\zeta\times\zeta'\big)(u_i,u_j):=\zeta(u_i)\zeta'(u_j)$ we have $\big(\zeta\times\zeta'\big)(u_i,u_i)=0$ for all $1\leq i\leq L$, and if $i\neq j$ then
 \begin{equation}\label{56}
     \big(\zeta\times\zeta'\big)(u_i,u_j)=\zeta(u_i)\zeta'(u_j)\leq\zeta(u_i)=a_i-m_i\leq |a_i-c_i|\leq\frac{\widetilde\varepsilon^p}{K^pL^2}.
\end{equation}
Now define $\pi\in\Pi(\nu,\nu')$ as follows: $\pi(u_i,u_j):=\big(\zeta\times\zeta'\big)(u_i,u_j)$ if $i\neq j$, and $\pi(u_i,u_i):=m_i$ for $1\leq i\leq L$.  Since $\varrho^p(u_i,u_i)\pi(u_i,u_i)=0$ for all $1\leq i\leq L$, using \eqref{56} we have the following upper bound for $d_p(\nu,\nu')$
\begin{equation}d_p(\nu,\nu')\leq\sqrt[p]{\sum_{1\leq i,j\leq L}\varrho^p(u_i,u_j)\cdot\pi(u_i,u_j)}<\sqrt[p]{L^2K^p \frac{\widetilde\varepsilon^p}{K^pL}}=\widetilde\varepsilon\leq\varepsilon.
\end{equation}

\end{proof}

\section{Wasserstein spaces with prescribed isometry group}

A natural question was raised by K\H{o}nig in \cite{Konig}: \emph{which groups are isomorphic to the automorphism group of a graph?} We recall that an automorphism of a simple graph $G(X,E)$ is a permutation $f:X\to X$ such that for any two $x,y\in X$  the pair $\{x,y\}$ form an edge (i.e., belongs to $E$) if and only if the pair $\{f(x),f(y)\}$ also form an edge. The group of automorphisms will be denoted by $\mathrm{Aut}\big(G(X,E)\big)$.

Of course, one can replace graphs with other mathematical structures, for example with Wasserstein spaces, and ask the same question. Since in the metric context, automorphisms are in particular isometries, the corresponding question reads as follows: which groups are isomorphic to the isometry group of a Wasserstein space? Using some famous results in graph theory, the answer for countable groups is a corollary of Theorem \ref{t:main}. (For analogous results for autohomeomorphism groups see \cite[Theorem 7]{deGroot}.)

\begin{corollary}\label{c:aut}
Let $H$ be a countable group and $p\geq1$ any real number. Then there exists a metric space $(X,\varrho)$ such that $\mathrm{Isom}\big(\mathcal{W}_p(X),d_p\big)\cong H$.
\end{corollary}
\begin{proof}
As an extension of Frucht's theorem \cite{Frucht}, de Groot proved that every countable group $H$ is isomorphic to the automorphism group of a countable simple graph $G(X,E)$ (see \cite[comments on p.96]{deGroot}). Let $(X,\varrho)$ be the metric space associated to $G(X,E)$ and consider the $p$-Wasserstein space $\mathcal{W}_p(X)$. According to Theorem \ref{t:main}, $\mathrm{Isom}\big(\mathcal{W}_p(X),d_p\big)\cong\mathrm{Isom}(X,\varrho)$, and therefore it is enough to show that $\mathrm{Isom}(X,\varrho)\cong\mathrm{Aut}\big(G(X,E)\big)$. It will turn out that these groups are identical as a set with the same operation (composition), so the identity map is an isomorphism.

If $\psi\in\mathrm{Isom}(X,\varrho)$ then for any pair $x,y\in X$ we have $\varrho(x,y)=1$ if and only if $\varrho(\psi(x),\psi(y))=1$. Or equivalently, $x$ and $y$ are joined by an edge if and only if $\psi(x)$ and $\psi(y)$ are joined by an edge. Since $\psi$ is a bijection, this means that $\psi\in\mathrm{Aut}\big(G(X,E)\big)$. On the other hand, every $\psi\in\mathrm{Aut}\big(G(X,E)\big)$ induces a length preserving bijection on the set of all paths as follows: if we have a path of length $k$ along the sequence of distinct vertices $x_0=x,x_1,\dots,x_k=y$, then the sequence $\psi(x_0),\dots,\psi(x_k)$ determines a path of length $k$ between $\psi(x)$ and $\psi(y)$. And therefore, the shortest path distance of $x$ and $y$ must be the same as the shortest path distance of $\psi(x)$ and $\psi(y)$. Since an automorphism is a bijection by definition, we have that $\psi\in\mathrm{Isom}(X,\varrho)$.
\end{proof}

We remark that there is no uniqueness above, because there is no uniqueness in Frucht's and de Groot's theorems. In fact, Izbicki proved in \cite{Izbicki} that there are uncountably many infinite graphs realizing any finite symmetry group. We also remark that although de Groot's theorem is valid for non-countable groups as well, we do not know the smallest possible order of the representing graph. However, it is important to note that if the cardinality of the vertex set is bigger than $\aleph_0$, then our method of proof does not work, therefore the following question remains open.
\begin{problem} Given an uncountable group $G$ and a fixed number $p\geq1$, does there exists a $p$-Wasserstein space whose isometry group is isomorphic to $G$?
\end{problem}


\end{document}